\documentclass[11pt]{article}

\usepackage[T1]{fontenc}
\usepackage{amsfonts,amssymb,amsmath,amsthm,amscd,enumerate}

\newtheorem{thm}{theorem}[section]
\newtheorem{theorem}[thm]{Theorem}
\newtheorem{proposition}[thm]{Proposition}
\newtheorem{lemma}[thm]{Lemma}
\newtheorem{corollary}[thm]{Corollary}
\newtheorem{remark}[thm]{Remark}

\DeclareMathOperator{\Supp}{Supp}
\DeclareMathOperator{\End}{End}
\DeclareMathOperator{\diag}{diag}

\begin{document}

\title{Degree-inverting involutions on matrix algebras}

\author{
Luís Felipe Gonçalves Fonseca
\thanks{\texttt{luisfelipe@ufv.br}}
\\
\\
Instituto de Ciências Exatas e Tecnlógicas \\
Universidade Federal de Viçosa \\
Florestal, MG, Brazil
\\
\\
Thiago Castilho de Mello \thanks{\texttt{tcmello@unifesp.br}}
\thanks{supported by Fapesp grant No. 2014/10352-4, Fapesp grant No.2014/09310-5, and CNPq grant No. 461820/2014-5. }
\\
\\
Instituto de Ci\^encia e Tecnologia\\
Universidade Federal de S\~ao Paulo\\
S\~ao Jos\'e dos Campos, SP, Brazil}

\maketitle

\begin{abstract}
	Let $F$ be an algebraically closed field of characteristic zero, and $G$ be a finite abelian group. If $A=\oplus_{g\in G} A_g$ is a $G$-graded algebra, we study degree-inverting involutions on $A$, i.e., involutions $*$ on $A$ satisfying $(A_g)^*\subseteq A_{g^{-1}}$, for all $g\in G$. We describe such involutions for the full $n\times n$ matrix algebra over $F$ and for the algebra of $n\times n$ upper triangular matrices.

\end{abstract}

\section*{Introduction}

Gradings on algebras have been intensively studied in the last two decades, in the associative and non-associative cases. One of the main purposes of such studies was to obtain a description of all possible gradings on a specific algebra or a class of algebras up to graded isomorphisms. As examples, gradings on finite dimensional simple associative, Lie and Jordan algebras have been studied since then.

The finite dimensional simple associative algebras over an algebraically closed field are exactly the matrix algebras over such field. The first results concerning gradings on matrix algebras were obtained by Knus \cite{Knus}.
Some other works on the next decades have appeared, but only in the end of the last century a systematic study of gradings on matrix algebras was started. Initially with a paper of D\u asc\u alescu, et. al \cite{Dascalescu} and later with other authors and papers, e.g.,\cite{BahturinSegalZaicev,BahturinZaicev, BarascuDascalescu, Boboc}.

Gradings on matrix algebras have a wide range of applications. The special case of superalgebras, i.e., algebras graded by the cyclic group of order two, has been extensively studied and was applied to the study of identities of associative algebras, in particular to the solution of the Specht problem for associative algebras over fields of characteristic zero \cite{Kemer}. Since then, the study of graded identities on associative and non-associative algebras has become of interest of many authors and an extensive number of papers about this subject has been written.

The knowledge of gradings on matrix algebras was also applied to obtain gradings on simple finite dimensional Lie and Jordan algebras See, for example, \cite{BahturinShestakovZaicev,Elduque,ElduqueKotchetov}. For this purpose, the authors in \cite{BahturinShestakovZaicev} considered involutions on matrix algebras and obtained gradings on the Lie algebra of skew-symmetric elements and on the Jordan algebra of symmetric elements. This was shown to be possible if and only if such involutions act on homogeneous elements by preserving their degrees. Following this idea, in the mentioned paper the authors describe gradings on a matrix algebra with involution, satisfying such property, what they call \emph{involution preserving gradings}.

When dealing with an elementary grading in a matrix algebra, one can easily see that the transpose involution inverts degrees of homogeneous elements.
In the present paper, we study a problem related this fact. We are interested in describing the involutions on a $G$-graded matrix algebra acting on homogeneous elements by inverting their degrees. We call such involutions \emph{degree-inverting involutions}. The approach to such problem is similar to that of \cite{BahturinShestakovZaicev}, i.e., we use the decomposition of gradings on matrix algebras as a tensor product of an elementary-graded matrix algebra by a fine-graded matrix algebra and we study each case separately. Nevertheless we use other methods and results, such as the graded version of the Skolem-Noether Theorem.

For the algebra of $n\times n$ upper triangular matrices, the only possible $G$-gradings are the elementary ones as proved in \cite{Valenti-Zaicev0}, and any involution on $M_n(F)$ can be obtained by a conjugation of the involution that reflects elements along the secondary diagonal as showed in \cite{DiVKoshLaS}. The approach in this case is by the well-known duality between $G$-gradings and $\hat G$-actions, where $\hat G$ is the dual group of $G$.

The paper is organized as follows.

Section 1 is devoted to the description of the well-known duality between bilinear forms and involutions on matrix algebras. In Section 2 we recall the description of gradings on matrix algebras. In Section 3 we present some results of \cite{BahturinShestakovZaicev} and \cite{BahturinZaicev} about involutions preserving gradings. Some of this results are used in Section 4, where we present our main result (Theorem \ref{maintheorem}) on the description of degree-inverting involutions, and the necessary background to achieve it. Finally, on Section 5, we deal with the same problem for the algebra of $n\times n$ upper triangular matrices. The mentioned results are valid for algebraically closed fields of characteristic zero and for a finite abelian group $G$.
 
\section{Involutions on matrix algebras and bilinear forms}

In this section we recall the duality between involutions on the $n\times n$ matrix algebra over the field $F$, $M_n(F)$, and nonsingular bilinear forms on the vector space $V=F^n$. We identify $M_n(F)$ with $\End (V)$.

An involution on an algebra $A$ is an antiautomorphism of order two, i.e., a linear map $^*:A\longrightarrow A$ satisfying $(ab)^*=b^*a^*$ and $(a^*)^*=a$, for all $a$, $b\in A$.

Let $V^{*}$ be the dual space of the vector space $V$. A bilinear form \linebreak $b: V \times V \rightarrow F$ is called nonsingular if the induced map 
\[\begin{array}{ccccccc}
\widehat{b}: &  V & \rightarrow & V^{*}\\
&  x & \mapsto     & \widehat{b}(x):& V & \longrightarrow & F\\
& & &  & y & \longmapsto & b(x,y)
\end{array}\] is an
isomorphism of linear spaces. 

Let $f: V \rightarrow V$ be a linear
transformation. Recall the transpose of $f$, $f^{t}: V^{*} \rightarrow
V^{*}$ is a linear application defined by:
$f^{t}(\phi)= \phi\circ f$ for all $\phi \in V^{*}$.

For a non-singular bilinear form $b:V \times V \rightarrow F$, we define the following map
\[ \begin{array}{cccc}
\sigma_{b}: & End_{F}(V)  &\longrightarrow & End_{F}(V) \\
&  f          &\longmapsto     & \widehat{b}^{-1}\circ f^{t}\circ \widehat{b}
\end{array}\]

One can observe that $\sigma_{b}$ is an antiautomorphism of
$End_{F}(V)$. The map $\sigma_{b}$ is called \emph{adjoint 	antiautomorphism} with respect to the nonsingular
bilinear form $b$. The map $\sigma_b$ is clearly F-linear . Alternatively, it is well known that $\sigma_{b}$ can be defined by:

\[b(x,f(y)) = b(\sigma_{b}(f)(x),y) \text{, for all } x,y \in V\]

Next, we present a well known result whose proof can be found, for instance, in \cite{InvolutionsBook}.

\begin{proposition}
	The map which associates to each nonsingular bilinear form $b$ on
	$V$ its adjoint antiautomorphism $\sigma_{b}$ induces a one-to-one
	correspondence between equivalence classes of nonsingular bilinear
	forms on $V$ modulo multiplication by a factor in $F^{*}$ and linear
	antiautomorphism of $End_{F}(V)$. Under this correspondence,
	$F$-linear involutions on $End_{F}(V)$ correspond to non-singular
	bilinear forms which are either symmetric os skew-symmetric.
\end{proposition}

\begin{corollary}
	Let $*$ be an involution on $M_{n}(F)$ and let $b: F^{n} \times F^{n}
	\rightarrow F$ be the non-singular bilinear form corresponding to $*$. If $\Phi$ is the matrix representation of $b$ with respect to some fixed linear basis of $V$, then in the matrix form, $*$ can be written as 
	\[X^{*} = \Phi^{-1}X^{t}\Phi.\]
	Moreover, $\Phi$ is uniquely defined by $*$ up to a scalar factor.
\end{corollary}

Let $\Phi$ be the matrix associated with the non-singular bilinear form $\varphi$, and $\Psi$ be the matrix associated with the non-singular bilinear form $\psi$. Then $\Phi \otimes \Psi$ is the matrix associated to the bilinear form $\varphi \otimes \psi$.

Now, consider $R = M_{n}(F)$ with involution $*:R \rightarrow R$
written as a tensor $R = R_{1}\otimes R_{2}$, where $R_{1} =
M_{k}(F),R_{2} = M_{l}(F)$, $R_{1}^{*} = R_{1}$ and $R_{2}^{*} =
R_{2}$. It was mentioned above, there exist matrices $\Phi$ and
$\Psi$ such that for all $X\in M_k(F)$ and $Y\in M_l(F)$, 
\[X^{*} = \Phi^{-1}X^{t}\Phi \quad \text{ and } \quad
Y^{*} = \Psi^{-1}Y^{t}\Psi.\]

Let $A \in M_{k}(F), B \in M_{l}(F)$. It is known that the
involution of Kronecker product $(A \otimes B)^{*}$ is $A^{*}\otimes
B^{*}$. Hence,
\begin{align*}
(A \otimes B)^{*} & = (\Phi_{1}^{-1}A^{t}\Phi_{2}^{-1}) \otimes
(\Phi_{2}^{-1}B^{t}\Phi_{2}) \\
& = (\Phi_{1}^{-1}\otimes
\Phi_{2}^{-1})(A \otimes B)(\Phi_{1}\otimes\Phi_{2}) \\
& =	(\Phi_{1}\otimes \Phi_{2})^{-1}(A \otimes B)(\Phi_{1}\otimes
\Phi_{2}),
\end{align*}
and it follows that the involution $*$ corresponds to the bilinear form $\overline \varphi$ on $F^n$ with matrix $\Phi_{1} \otimes \Phi_{2}$. And $\overline \varphi$ is the tensor product of bilinear forms $\varphi$ and $\psi$.  

A combination of the above remarks give us Lemma 1 of \cite{BahturinShestakovZaicev}:

\begin{lemma}
	Let $R = M_{n}(F)$ be a matrix algebra with an involution $*: R
	\longrightarrow R$. Let $R = R_{0}\otimes R_{1} \otimes \ldots \otimes
	R_{k}$ where $R_{i} = M_{n_{i}}(F)$ and $R_{i}^{*} = R_{i}, i =
	0,\ldots,k$. Then there exist non-degenerate symmetric or skew-symmetric bilinear forms $\varphi_{0},\ldots,\varphi_{k}$ on $F^{n_0}$, $\dots$, $F^{n_k}$, respectively, such that, for each $i=0,1,\dots, k$, the restriction of $*$ to $R_i$ corresponds to $\varphi_i$, $i=0,1,\dots,k$, and the involution $*$ on $R$ corresponds to $\varphi=\varphi_0\otimes\cdots \otimes \varphi_k$.
\end{lemma}

\section{Gradings on matrix algebras}

In this section we recall basic results on the description of group gradings on $M_n(F)$.

If $A$ is an $F$-algebra and $G$ is a group, a $G$-grading on $A$ is a decomposition of $A$ as a direct sum of subspaces $A=\oplus_{g\in G}A_g$, indexed by elements of the group $G$, which satisfy $A_gA_h\subseteq A_{gh}$, for any $g, h\in G$. If $a\in A_g$, for some $g\in G$, we say that $a$ is homogeneous of degree $g$ and we denote $\deg (a)=g$. The support of the grading, is the subset of $G$, $\text{Supp}(A)=\{g\in G\,;\, A_g\neq \{0\}\}$. 

If $1\leq i,j\leq n$, we denote by $e_{ij}$ the matrix with $1$ on the position $(i,j)$, and and $0$ elsewhere. We call them \emph{elementary matrices}. For each $g\in G$, let $R_g\subseteq M_n(F)$ be the subspace generated by the elementary matrices $e_{ij}$ for $i$ and $j$ satisfying $g_{i}^{-1}g_{j}=g$.
Then $M_n(F)=\oplus_{g\in G} R_g$ is a $G$-grading on $M_n(F)$ called \emph{elementary grading defined by $(g_1,\dots,g_n)$}.

We recall a known result from \cite{Dascalescu}.

\begin{theorem}
	If $G$ is any group, a $G$-grading of $M_{n}(F)$ is elementary if and only if all matrix units $e_{ij}$ are homogeneous.
\end{theorem}

Recall that the elementary $G$-grading defined by $(g_1,\dots,g_n)\in G^n$ on the algebra $M_n(F)$ can be regarded as the induced grading on the algebra of all linear endomorphisms of a $G$-graded vector space $V=\oplus_{g\in G}V_g$, of dimension $n=\sum \dim(V_g)$. The elements of $V_g$ are called homogeneous of degree $g$, and if $v\in V_g$, we write $\deg(v)=g$. A linear transformation $f:V\longrightarrow V$ is called homogeneous of degree $h$ if $f(V_g)\subseteq V_{hg}$, for all $g$.
Let $h_1,\dots,h_k$ be all distinct elements of the group $G$ among $g_1,\dots,g_n$.
If $\pi_g:V\longrightarrow V_g$ is the canonical projection then $\pi_gf\pi_h$ is a homogeneous transformation of degree $gh^{-1}$ for any $f\in \End(V)$ and $f=\sum_{g,h\in G}\pi_gf\pi_h$ is the decomposition of $f$ into its homogeneous components. By Fixing a homogeneous basis of $V$, one obtains an isomorphism $\End V\longrightarrow M_n(F)$ of graded $F$-algebras. If one chooses distinct homogeneous basis of $V$, one gets graded automorphisms of $M_n(F)$. In order to obtain an elementary $G$-grading on $M_n(F)$ induced by $(g_1,\dots,g_n)$, it is enough to  take $V=V_1\oplus \cdots\oplus V_n$ with $\deg V_i=g_i^{-1}$, for $i\in\{1,\dots,n\}$.

Another important kind of grading is the so-called fine grading. A G-grading on an algebra $A=\oplus_{g\in G}A_g$ is a \emph{fine grading} if $\dim A_g\leq 1$, for all $g\in G$.

In \cite{BahturinZaicev}, the authors prove the following result about fine gradings.

\begin{theorem}
	If $M_n(F)$ is endowed with a fine $G$-grading, then $\Supp(R)$ is a subgroup of $G$. Furthermore, all its non-zero homogeneous elements are invertible.
\end{theorem}

For an algebraically closed field, $F$, a particular case of a fine grading on $M_n(F)$ is the so called $\varepsilon$-grading.
If $\varepsilon\in F$ is a primitive $n$-th root of 1, define the following $n \times n$ matrices over $F$:

\begin{equation}\label{XaXb}
X_a = \begin{pmatrix} \varepsilon^{n-1} & 0 & \ldots & 0 & 0 \\
0 & \varepsilon^{n-2} & \ldots & 0 & 0 \\
\vdots & \vdots &\ddots &\vdots & \vdots\\
0 & 0 & \ldots & \varepsilon & 0 \\
0 & 0 & \ldots & 0& 1 \\
\end{pmatrix} \quad
\text{and} \quad
X_b = \begin{pmatrix} 0 & 1 & 0 & \ldots & 0 \\
0 & 0 & 1 & \ldots & 0 \\
\vdots & \vdots &\vdots &\ddots & \vdots\\
0 & 0 & 0 & \ldots & 1 \\
1 & 0 & 0 & \ldots & 0
\end{pmatrix}
\end{equation}

It is easy to see  $\{X_a^{i}X_b^{j}| 1 \leq i,j \leq n\}$ is a linear basis of $M_{n}(F)$ and that all products of these basis elements are uniquely defined by the relations
\[\begin{array}{lcr} X_aX_b = \varepsilon X_bX_a & \text{ and } & X_a^{n} = X_b^{n} = I\end{array}\]

Now, consider $G = \mathbb{Z}_{n}\times \mathbb{Z}_n$. For $g = (\overline{i},\overline{j}) \in G$, we set $C_{g} = A^{i}B^{j}$ and denote $R_{g} = span_{F}\{C_{g}\}$. In light of the above relations, $R = \bigoplus_{g\in G} R_{g}$ is a $G$-grading on $R$. This grading is called \emph{$\varepsilon$-grading}.

The next construction allows one to induce a $G$-grading on the tensor product of a $G$-graded algebra $A$, with the algebra $M_n(F)$ endowed with an elementary $G$-grading.

Let $A$ be any $G$-graded algebra, $B = M_{n}(F) = \bigoplus_{g \in G}B_g$ be the matrix algebra over $F$ endowed with an   elementary grading induced by $\overline{g} = (g_{1},\ldots,g_{n})\in G^{n}$. So, direct computations show that $A\otimes B=\oplus_{g\in G}(A\otimes B)_g$ is a $G$-grading on $A\otimes B$ if one sets
\[(A \otimes B)_{g} = \text{span}_F \{a \otimes e_{ij}| a \in A_{h},
g_{i}^{-1}hg_{j} = g\}.\]

The grading defined above is called \emph{induced}.

We now recall results of Bahturin, Sehgal and Zaicev \cite[Theorems 5 and 6]{BahturinSegalZaicev}, on the description of abelian gradings on matrix algebras over an algebraically closed field of characteristic zero.

\begin{theorem}\label{gradings}
	Let $R = M_{n}(F) = \bigoplus_{g \in G} R_{g}$ be a matrix algebra  over an algebraically closed field $F$, of characteristic zero, graded by a finite abelian group $G$. Then there exists a decomposition $n = kl$, a subgroup $H \subseteq G$ of order $k^{2}$ and a $l$-tuple $\overline{g} = (g_{1},\ldots,g_{l}) \in G^{l}$ such that $M_{n}(F)$ is isomorphic as a $G$-graded algebra to the tensor
	product $M_{k}(F) \otimes M_{l}(F)$ with an induced $G$-grading where $M_{k}(F)$ is a $H$-graded algebra with fine $H$-grading and $M_{l}(F)$ is endowed with an elementary grading determined by $\overline{g}$. Moreover, $H$ decomposes as $H\cong H_1\times \cdots \times H_t$, $H_i\cong \mathbb{Z}_{n_i}\times \mathbb{Z}_{n_i}$ and $M_k(F)$ is isomorphic to $M_{n_1}(F)\otimes\cdots \otimes M_{n_t}(F)$ as an $H$-graded algebra, where $M_{n_i}(F)$ is an $H_i$-graded algebra with some $\varepsilon_i$-grading.
\end{theorem}

\section{Gradings and Involutions on matrix algebras}

We now focus on $M_n(F)$ with both graded and involution structures.

In the papers \cite{BahturinShestakovZaicev,BahturinZaicev2} the authors describe what they call \emph{involution preserving gradings} (or \emph{graded involutions}) on the algebra $R=M_n(F)$ with involution $*$, (Theorem \ref{bahturin-Zaicev} below) i.e., gradings $R=\oplus_{g\in G} R_g$ on $R$, satisfying $(R_g)^{*}\subseteq R_g$, for all $g\in G$. Such  description has important applications on classifications of group gradings in some classes of simple finite dimensional Lie and Jordan algebras. We recall such description here.

\begin{lemma}\label{Lemma1}
	Let  $G$ be an abelian group and $R=M_n(F)$ be a $G$-graded matrix algebra with an elementary grading, defined by an $n$-tuple $(g_1,\dots,g_n)\in G^n$. Let $*$ be an involution on $R$, defined by a symmetric non-degenerate bilinear form. 
	If $*$ is a graded involution on $R$, then there exists integers $m$ and $l$ such that $n=2l+m$, and after a renumbering, $g_1g_{l+1}=g_2g_{l+2}=\cdots=g_lg_{2l}=g_{2l+1}^2=\cdots=g_n^2$. The involution $*$ acts as $X^*=\Phi^{-1}X^t\Phi$, where 
		\[\Phi=\begin{pmatrix}
		0   & I_l & 0\\
		I_l & 0   & 0\\
		0   & 0   & I_m\\
		\end{pmatrix}.\]
		and $I_s$ denotes the $s\times s$ identity matrix.
\end{lemma}

\begin{lemma}\label{Lemma2}
	Let  $G$ be an abelian group, $n=2l$ be an even positive integer and $R=M_n(F)$ be a $G$-graded matrix algebra with an elementary grading, defined by an $n$-tuple $(g_1,\dots,g_n)\in G^n$. Let $*$ be an involution on $R$, defined by a skew-symmetric non-degenerate bilinear form. 
	If $*$ is a graded involution on $R$ then, after a renumbering, $g_1g_{l+1}=g_2g_{l+2}=\cdots=g_lg_{2l}$. The involution $*$ acts as $X^*=\Phi^{-1}X^t \Phi$, where
		\[\Phi=\begin{pmatrix}
		0   & I_l \\
		-I_l & 0  \\
		\end{pmatrix},\]
		and $I_l$ denotes the $l\times l$ identity matrix.
\end{lemma}

\begin{lemma}\label{Lemma3}
	Let $R=M_2(F)$ be a $2\times 2$ matrix algebra endowed with an involution $*:R\longrightarrow R$ corresponding to a symmetric or skew symmetric non-degenerate bilinear form with matrix $\Phi$ and with the $(-1)$-grading of $M_2(F)$ by the group $G=\langle a \rangle \times \langle b \rangle$. Then $*$ is a graded-involution if and only if $\Phi$ can be chosen to be one of the matrices $I$, $X_a$, $X_b$ or $X_aX_b$.
\end{lemma}

Now the description of the graded involutions on $M_n(F)$.

\begin{theorem}\label{bahturin-Zaicev}
	Let $R=M_n(F)=\oplus_{g\in G}R_g$ be a matrix algebra over an algebraically closed field of characteristic zero graded by the group $G$ and $\Supp(R)$ generates $G$. Suppose that $*:R\longrightarrow R$ is a graded involution on $R$. Then $G$ is abelian and $R$ is isomorphic as a $G$-graded algebra to the tensor product $R^{(0)}\otimes R^{(1)}\otimes \cdots \otimes R^{(k)}$ of a matrix subalgebra $R^{(0)}$ with elementary grading  and $R^{(1)}\otimes \cdots \otimes R^{(k)}$ a matrix subalgebra with fine grading. Suppose further that both these subalgebras are invariant under the involution $*$. Then $n=2^km$ and
	\begin{enumerate}[(1)]
		\item $R^{(0)}=M_m(F)$ with a elementary grading defined by $\overline{g}=(g_1,\dots,g_m)$, as in Lemma $\ref{Lemma1}$, if $*$ is symmetric or as in Lemma \ref{Lemma2} if $*$ is skew-symmetric.
		\item  $R^{(1)}\otimes \cdots \otimes R^{(k)}$ is a $T=T_1\times \cdots \times T_k$-graded algebra and any $R^{(i)}\cong M_2(F)$ is $T_i\cong \mathbb{Z}_2\times \mathbb{Z}_2$-graded algebra as in Lemma \ref{Lemma3}.
	\end{enumerate}
\end{theorem}

\section{Degree-inverting involutions}

An involution $*$ on a $G$-graded algebra $A=\oplus_{g\in G}A_g$ is called a \emph{degree-inverting involution} if it satisfies for all $g\in G$, $(A_g)^*\subseteq A_{g^{-1}}$. One can easily verify that the above condition is equivalent to $(A_g)^*=A_{g^{-1}}$, for all $g\in G$.

This section is devoted to describing analogous results to those in section 3, in the case of degree-inverting involutions.

The next lemma shows that the transpose involution in $M_n(F)$ is a degree-inverting involution, when it is endowed with an elementary grading.

\begin{lemma}
	Let $R=M_n(F)$, $n\geq 2$ and let $t:R\longrightarrow R$ be the transpose involution on $R$. If $G$ is a group and $R$ has an elementary $G$-grading, then for all $g\in G$, $(A_g)^t\subseteq A_{g^{-1}}$.
\end{lemma}

\proof Suppose such grading is defined by an $n$-tuple $(g_1,\dots,g_n)$ of elements of $G$. And write $R=\oplus_{g\in G} R_g$. Of course it is enough to prove the lemma for elementary matrices.
Let $e_{ij}\in R_g$. Then $g=g_i^{-1}g_j$.
Clearly $e_{ij}^t=e_{ji}$ and $deg(e_{ji})=g_{j}^{-1}g_i=g^{-1}$. \endproof

Now a natural problem arises: ``Describe the degree-inverting involutions on $M_n(F)$''.  Such description is the main purpose of this paper. The results obtained here are similar to those obtained in \cite{BahturinShestakovZaicev} and \cite{BahturinZaicev2} for the case of graded-involutions.

The next result handles the case of elementary gradings on $M_n(F)$ and an involution $*$ on $M_n(F)$ satisfying $(R_g)^*\subseteq R_{g^{-1}}$.

In order to handle the case of elementary gradings, we are going to need a graded version of the well-known Skolem-Noether theorem. From the classical Skolem-Noether theorem, we obtain that any automorphism of $M_n(F)$ is inner, i.e., given by conjugation by an invertible matrix.

The graded version of such theorem was proved by Hwang and \linebreak Wadsworth in \cite{hw} and we recall it here. For more information concerning graded structures, simple and semisimple graded algebras, we refer to \cite{tw}. In what follows, we denote by $C_A(B)$ the centralizer of $B$ in $A$ and by $Z(A)$ the center of the algebra $A$.

\begin{theorem}
	Let $F$ be a graded field and $A$ be a central simple graded $F$-algebra. Let $B$ be a graded simple $F$-subalgebras of $A$, $C=C_A(B)$ and $Z=Z(B)$. Let $\alpha: B\longrightarrow A$ be a graded $F$-algebra homomorphism. Then,
	\begin{enumerate}
		\item There exists an invertible element $a\in A$ such that $\alpha(b)=a^{-1}ba$, for all $b \in B$.
		\item If $C$ is a division ring, then the element $a$ of part (1) can be chosen to be homogeneous in $A$.
		\item The element $a$ of the above item can be chosen to be homogeneous if and only if there is a graded homomorphism $\gamma:C\longrightarrow A$ such that $\gamma_{|Z}=\alpha_{|Z}$ and $\gamma(B)$ centralizes $\alpha(B)$ in $A$.
	\end{enumerate}
\end{theorem}

In particular, we have the following corollary:

\begin{corollary}
	Let $A$ be a central simple graded $F$-algebra. Then for \linebreak every graded $F$-algebra automorphism $\phi$ of $A$, there exists an invertible homogeneous element  $a\in A$, such that $\phi(x)=a^{-1}xa$, for all $x\in A$.
\end{corollary}

Since $M_n(F)$ is a central simple algebra, it is a central simple graded algebra, and the following is a particular case of the former.

\begin{corollary}\label{Skolem}
	Let $F$ be a field and $\phi$ be a graded automorphism of $M_n(F)$. Then there exist an invertible homogeneous matrix $P\in M_n(F)$, such that $\phi(A)=P^{-1}AP$, for all $A\in M_n(F)$.
\end{corollary}

The next result describes degree-inverting involutions on $M_n(F)$ with an elementary grading.

\begin{proposition}\label{elementary}
	
	Let $R = M_n(F)$ be a matrix algebra with an  involution $*:R\longrightarrow R$, defined by a non-degenerate bilinear form $\varphi$.  Let $G$ be an abelian group and $R=\oplus_{g\in G}R_g$ be an elementary grading on $R$. If $*$ is a degree-inverting involution on $R$, then $R$, as a graded algebra with involution, is isomorphic to $M_n(F)$ with an elementary $G$-grading defined by an $n$-tuple $(g_1, \dots ,g_n)$ and with an involution $X\longmapsto X^*=\Phi^{-1}X^t\Phi$, where
	\begin{enumerate}[(1)]
		\item $n=2l+m$, for some $l$, $m\in \mathbb{N}$, and 
		\begin{equation}\label{symmetric}
		\Phi=\begin{pmatrix} 
		0   & I_l & 0\\
		I_l & 0   & 0\\
		0   & 0   & I_m\\
		\end{pmatrix},
		\end{equation}
		if $\varphi$ is symmetric. Moreover, 
		\begin{enumerate}
			\item if $m=0$, then after a renumbering,
			$g_1g_{l+1}^{-1}=\cdots=g_lg_{2l}^{-1}$, and  $g_i^2=g_{i+l}^2$, for all $i\in\{1,\dots,l\}$;
			\item if $l\neq 0$ and $m\neq 0$, then after a renumbering $g_1g_{l+1}^{-1}=\cdots=g_kg_{2l}^{-1}$, $g_1^2=g_{2}^2=\cdots= g_{2l}^2$, and $g_1g_{l+1}=g_2g_{l+2}=\cdots = g_lg_{2l}=g_{2l+1}^2=\cdots = g_{m}^2$;
			\item if $l=0$, $*$ is the transpose involution and $g_1,\dots, g_m$ are arbitrary.
		\end{enumerate}
		\item $n=2l$, for some $l\in \mathbb{N}$, and
		\begin{equation}\label{skew}
		\Phi=\begin{pmatrix}
		0   & I_l\\
		-I_l & 0  \\
		\end{pmatrix}
		\end{equation}
		 if $\varphi$ is skew-symmetric. Moreover, after a renumbering $g_1g_{l+1}^{-1}=\cdots=g_lg_{2l}^{-1}$, and $g_i^2=g_{i+l}^2$, for all $i\in\{1,\dots,l\}$.
	\end{enumerate}
\end{proposition}

\begin{proof}
	Let now $h_1,\dots,h_k$ be all distinct elements of $G$ among $g_1,\dots,g_n$. Permuting $g_1,\dots,g_n$, i.e., changing the basis of $V$, we chose the isomorphic copy of $M_n(F)$ such that 
	\[(g_1,\dots,g_n)=(h_1,\dots h_1,h_2,\dots,h_2,\dots,\dots,h_k\dots,h_k)\]
	Each $h_i$ above appears $m_i$ times, and 
	\[V=V_{h_1^{-1}}\oplus\cdots \oplus V_{h_k^{-1}}\]
	with $\deg(V_{h_i^{-1}})=h_i^{-1}$ and $\dim(V_{h_i^{-1}})=m_i$, for $i\in\{1,\dots,k\}$. Fixing any basis in $V_{h_1^{-1}}, \dots, V_{h_k^{-1}}$ we obtain an elementary grading on $M_n(F)$ isomorphic to the initial one, such that any $A\in M_n(F)$ decomposes in $k^2$ blocks \[M=\begin{pmatrix}
	M_{11}  & \cdots &  M_{1k}\\
	\vdots & \ddots & \vdots\\
	M_{k1} & \cdots & M_{kk}\
	\end{pmatrix}\] 
	where $M_{ij}$ is of order $m_i\times m_j$ and the matrix units of this blocks are of degree $h_i^{-1}h_j$.
	
	Now let us consider $*:M_n(F)\longrightarrow M_n(F)$ an involution satisfying $((M_n(F))_g)^*\subseteq (M_n(F))_{g^{-1}}$. Let $\Phi\in M_n(F)$ be the matrix of the bilinear form defining $*$. Write as above \[\Phi=\begin{pmatrix}
	\Phi_{11}  & \cdots &  \Phi_{1k}\\
	\vdots & \ddots & \vdots\\
	\Phi_{k1} & \cdots & \Phi_{kk}\
	\end{pmatrix}\]
	Now we observe that if $X\in M_n(F)$ is homogeneous of degree $g$, then $X^t$ is homogeneous of degree $g^{-1}$. Since $(M_n(F)_g)^*\subseteq M_n(F)_{g^{-1}}$ for each $g$, we obtain for all $X\in M_n(F)$,  \[\deg(\Phi^{-1}X^t\Phi)=\deg(X^*)=\deg(X)^{-1}=\deg(X^t).\] In particular, applying the above for $X^t$, we obtain that $\deg(\Phi^{-1}X\Phi)=\deg(X)$, for all $X\in M_n(F)$. Hence the map
	\[\begin{array}{rcl}
	M_n(F) & \longrightarrow & M_n(F)\\
	X      & \longmapsto     & \Phi^{-1}X\Phi\\
	\end{array}\]
	is a graded automorphism of $M_n(F)$. In particular, it follows from Corollary \ref{Skolem} that $\Phi$ can be chosen to be homogeneous. As a consequence, since the $h_i$ are pairwise distinct and $\Phi$ is invertible, in the block-description of $\Phi$, one obtains that in each block-row and in each block-column, there is exactly one nonzero block. If $\Phi_{ii}\neq 0$, since $\varphi$ is non-degenerate,  changing the basis of $V_{h_i^{-1}}$ we can assume that $\Phi_{ii}=I_{m_i}$, the identity matrix of size $m_i$, if $\varphi$ is symmetric or \[\Phi_{ii}=\begin{pmatrix}
	0 & I \\ 
	-I& 0 \\
	\end{pmatrix}\] 
	if $\varphi$ is skew-symmetric. 
	
	If $\Phi_{ii}=0$, since $\Phi$ is symmetric or skew-symmetric  and non-degenerate, permuting $h_1,\dots, h_k$, we can assume that $m_i=m_{i+1}$ and $\Phi_{i i+1}=I$. As a consequence, $\Phi_{i+1 i}=I$, if $\varphi$ is symmetric, and $\Phi_{i+1 i}=-I$, if $\varphi$ is skew symmetric. Choosing an appropriate homogeneous basis of $F^n$, we obtain that 
	\[\Phi=\begin{pmatrix}
	0   & I_l & 0\\
	I_l & 0   & 0\\
	0   & 0   & I_m\\
	\end{pmatrix}\]
	if $\varphi$ is symmetric or 
	\[\Phi=\begin{pmatrix}
	0   & I_l\\
	-I_l & 0  \\
	\end{pmatrix}\]
	if $\varphi$ is skew-symmetric.
	
	Suppose now $\varphi$ skew-symmetric. Using that $\Phi$ is of the form (\ref{skew}), for each $i$ and $j$, we have $e_{ij}^*=\pm e_{j-l,i+l}$, where the sum on the indexes are taken modulo $2l=n$. Hence, $\deg(e_{ij}^*)=g_{j-l}^{-1}g_{i+l}$. On the other hand, since $*$ inverts the degree, we obtain $\deg(e_{ij}^*)=g_j^{-1}g_i$. In particular, for each $i\leq l$ and for $j=1+l$,  
	\[g_{1+l}^{-1}g_i=g_1^{-1}g_{i+l}\]
	which is equivalent to \[g_1g_{1+l}^{-1}=g_{i+l}g_i^{-1}\]
	Hence, we obtain \[g_1g_{1+l}^{-1}=g_2g_{2+l}^{-1}=\cdots=g_lg_{2l}^{-1}.\]
	
	Applying $*$ on $e_{i, i+l}$ we obtain $e_{i,i+l}^*=\pm e_{i,i+l}$ and analyzing the degrees, we obtain for each $i$, \[g_{i+l}^{-1}g_i=g_i^{-1}g_{i+l}.\]
	Since $G$ is abelian, $g_i^2=g_{i+l}^2$, for all $i$.
	
	Now let us consider the case where $\varphi$ is symmetric and $\Phi$ is as in (\ref{symmetric}). If $i,j\leq 2l$, the same arguments above show that 
	$g_1g_{l+1}^{-1}=\cdots=g_lg_{2l}^{-1}$, $g_i^2=g_{i+l}^2$, for all $i\in\{1,\dots,l\}$.
	
	If $i,j\geq 2l+1$, we obtain $e_{ij}^*= e_{ji}$ and such equation implies is no restriction on the elements of $G$. In particular, if $l=0$, any $(g_1,\dots,g_n)\in G^n$ induces a grading satisfying $R_g^*\subseteq R_{g^{-1}}$. If $m=0$, the relations satisfied by $g_1,\dots,g_n$ are the same as in the skew-symmetric case.
	
	Finally, if $l\neq 0$ and $m\neq 0$, consider $i\leq 2l$ and $j\geq 2l+1$. Then $e_{ij}^*=e_{i+l,j}$, again, the first indexes are taken modulo $2l$. Hence for all $i\leq 2l$ and $j\geq 2l+1,$ \[g_{j}^{-1}g_1=g_{i+l}^{-1}g_j\] and we obtain $g_j^2=g_ig_{i+l}$. In particular, $g_1g_{l+1}=g_2g_{l+2}=\cdots= g_lg_{2l}=g_{2l+1}^2=\cdots=g_m^2$. Since we also have $g_1g_{l+1}^{-1}=g_2g_{l+2}^{-1}=\cdots= g_lg_{2l}^{-1}$, we obtain $g_1^2=g_2^2=\cdots= g_{2l}^2$.
\end{proof}

We now turn our attention to degree-inverting involutions on $M_n(F)$ graded by $\varepsilon$-gradings. The next lemma shows that there are no such degree-inverting involutions on $M_n(F)$, unless $n\leq 2$.

\begin{lemma}
	Let $R=M_n(F)=\bigoplus_{g\in G}R_g$, $n\geq 2$, be graded by an $\varepsilon$-grading. Let $\varphi:R\longrightarrow R$ be an antiautomorphism.
	If for all $g\in G=\mathbb{Z}_n\times \mathbb{Z}_n$, $\varphi(R_g)\subseteq R_{g^{-1}}$, then $n=2$ and $\varphi$ is given by $\varphi(X)=\Phi^{-1}X\Phi$, where $\Phi$ coincides with a scalar multiple of one of the matrices $I$, $X_a$, $X_b$ or $X_aX_b$.
\end{lemma}

\proof Any antiautomorphism $\varphi$ of $M_n(F)$ is of the form
\[\varphi(A)=\Phi^{-1}A^t\Phi,\]
for some nonsingular matrix $\Phi$.

Let $\varepsilon\in F$ be a primitive $n$-th root of 1.
We consider the matrices $X_a$ and $X_b$ as in (\ref{XaXb}). Let us analyze the action of $\varphi$ on $X_a$ and $X_b$.

Since $X_b^{t}=X_b^{-1}$, we have
\[\varphi(X_b)=\Phi^{-1}X_b^t\Phi=\Phi^{-1}X_b^{-1}\Phi.\]

Since $\varphi(R_g)\subseteq R_{g^{-1}}$, $\varphi(X_b)=\alpha X_b^{-1}$, for some $\alpha\in F^\times$.
Then, \[\Phi^{-1}X_b^{-1}\Phi=\alpha X_b^{-1}.\]

From the above equation and from the fact that $X_b^n=I_n$,  we obtain that $\alpha^n=1$, i.e., $\alpha= \varepsilon^i$ for some $i\in \{0,\dots, n-1\}$. Moreover,
\[X_b^{-1}\Phi X_b=\varepsilon^i\Phi.\]

Observe that $X_b$ defines a linear transformation of $M_n(F)$ by conjugation, given by
\[\begin{array}{cccc}
T: & M_n(F) & \longrightarrow & M_n(F)\\
&  A     & \longmapsto     & X_b^{-1}AX_b\\
\end{array},\]
and that $\Phi$ is an eigenvector of $T$ associated to the eigenvalue $\varepsilon^i$.

Let now $P$ be the linear span of $\{I, X_b, \dots, X_b^{n-1}\}$. Then \[R=P\oplus X_aP\oplus \cdots \oplus X_a^{n-1}P.\]

Since $X_b^{-1}X_aX_b=\varepsilon X_a$, for each $i$, $T$ acts in $X_a^iP$ as multiplication by $\varepsilon^i$. In particular, all eigenvectors associated to the eigenvalues $\varepsilon^i$, are in $X_a^i P$. Then $\Phi\in X_a^iP$, i.e., there exists $Q\in P$ such that \[\Phi=X_a^iQ.\]

Let us now analyze the action of $\varphi$ on $X_a$. Since for all $g\in G$,  $\varphi(R_g)\subseteq R_{g^{-1}}$, $\varphi(X_a)=\gamma X_a^{-1}$, for some $\gamma\in F^\times$. Then
\[\varphi(X_a)=\Phi^{-1}X_a\Phi=\gamma X_a^{-1},\] which implies that \[X_a\Phi X_a=\gamma \Phi.\]
Write $Q=\sum_{j=0}^{n-1}\alpha_jX_b^j$, $\alpha_j\in F$ and we obtain $\Phi=X_a^i\sum_{j=0}^{n-1}\alpha_jX_b^j$. Thus
\[X_a\Phi X_a=X_a^iX_a\sum_{j=0}^{n-1}\alpha_jX_b^jX_a=\sum_{j=0}^{n-1}\varepsilon^j \alpha_jX_a^iX_b^jX_a^2.\]
As a consequence,
\[\sum_{j=0}^{n-1}\gamma\alpha_jX_a^iX_b^j=\sum_{j=0}^{n-1}\varepsilon^j \alpha_jX_a^iX_b^jX_a^2.\]
Since $\Phi\neq 0$, there exist $j$ such that $\alpha_j\neq 0$.
Hence comparing the $G$-degree in the above equation, we obtain that $i+2=i \; (\hspace{-.2cm}\mod{n})$, i.e., $n=2$.

Now, since $\Phi= X_a^iQ$ with $Q=\alpha_0I+\alpha_1X_b$, we obtain 
\begin{equation}\label{alpha}
\Phi= \alpha_0X_a^i+\alpha_1 X_a^iX_b.
\end{equation}
Since $n=2$, $X_a^{-1}=X_a$,  $X_b^{-1}=X_b$ and the argument used above applies if we change $a$ and $b$. As a consequence, there exists $j\in\{0,1\}$, and $\beta_0,\beta_1\in F$, such that 
\begin{equation}\label{beta}
\Phi=\beta_0X_b^j+\beta_1X_b^jX_a.
\end{equation}
By comparing equations (\ref{alpha}) and (\ref{beta}) above, for all combinations of $1 \leq i,j\leq 2$, we obtain that $\Phi$ is one of $I$, $X_a$, $X_b$ or $X_aX_b$. \endproof

\begin{corollary}\label{corolario}
	Let $R=M_n(F)$, $n\geq 2$, be graded by an $\varepsilon$-grading and $*: R\longrightarrow R$ be an involution.
	If for all $g\in G=\mathbb{Z}_n\times \mathbb{Z}_n$, $R_g^*\subseteq R_{g^{-1}}$, then $n=2$.
\end{corollary}

\begin{remark}
	Since in the Klein group $\mathbb{Z}_2\times \mathbb{Z}_2$, each element is its own inverse, Lemma \ref{Lemma3} also holds for involutions $*$ satisfying $(R_g)^{*}\subseteq R_{g^{-1}}$.
	We restate this result in the new terminology.
\end{remark}

\begin{corollary}\label{Lemma3'}
		Let $R=M_2(F)$ be a $2\times 2$ matrix algebra endowed with an involution $*:R\longrightarrow R$ corresponding to a symmetric or skew symmetric non-degenerate bilinear form with matrix $\Phi$ and with the $(-1)$-grading of $M_2(F)$ by the group $G=\langle a \rangle \times \langle b \rangle$. Then $*$ is a degree-inverting involution if and only if $\Phi$ can be chosen to be one of the matrices $I$, $X_a$, $X_b$ or $X_aX_b$.

\end{corollary}

A combination of the above results and the theorem on the  classification of $G$-gradings on $M_n(F)$ (Theorem \ref{gradings}), gives us the main result of the paper, which describes the degree-inverting involutions on $M_n(F)$.

\begin{theorem}\label{maintheorem}
	Let $R=M_n(F)$ be a matrix algebra over an algebraically closed field of characteristic zero graded by a finite abelian group $G$. Suppose that $*$ is a degree-inverting involutionon $R$. Then, there is a graded isomorphism  \[R\cong R^{(0)}\otimes R^{(1)}\otimes  \dots \otimes R^{(k)},\] where $R^{(0)}$ is a subalgebra of $R$ with an elementary grading and $R^{(1)}\otimes \dots \otimes R^{(k)}$ a subalgebra of $R$ with a fine grading. If both subalgebras $R_0$ and $R^{(1)}\otimes \cdots\otimes R^{(k)}$ are invariant under the involution $*$, then $n=2^km$ and
	\begin{enumerate}[(1)]
		\item $R^{(0)}=M_m(F)$, with an elementary $G$-grading defined by an $m$-tuple $\overline{g}=(g_1,\dots,g_m)$ of elements of
		$G$. The involution $*$ acts on $M_m(F)$ as $X^*=\Phi^{-1}X^t\Phi$, where $\Phi$ and the elements $g_1,\dots,g_m$ are as in Proposition \ref{elementary}.
		\item  $R^{(1)}\otimes \cdots \otimes R^{(k)}$ is a $T=T_1\times \cdots \times T_k$-graded algebra and any $R^{(i)}\cong M_2(F)$ is $T_i\cong \mathbb{Z}_2\times \mathbb{Z}_2$-graded algebra. The involution $*$ acts on $R^{(1)}\otimes \cdots \otimes R^{(k)}$ as in Corollary \ref{Lemma3'}.
	\end{enumerate}
\end{theorem}

\section{Degree-inverting involutions on upper triangular matrices}

We now turn our attention to the algebra of upper triangular matrices over $F$, $UT_n(F)$, where $F$ is an algebraically closed field of characteristic zero. 

The graded involutions on $UT_n$ have been described in \cite{Valenti-Zaicev}. In that paper, the authors use the description of gradings and involutions on $UT_n$ given in \cite{Valenti-Zaicev0} and \cite{DiVKoshLaS} respectively. We recall such results here. First, the description of gradings.

\begin{theorem}\cite[Theorem 7]{Valenti-Zaicev0}\label{gradingUTn}
	Let $G$ be an arbitrary group and $F$ a field. Suppose that the algebra $UT_n(F)=R=\oplus_{g\in G} R_g$ is $G$-graded. Then $R$, as a $G$-graded algebra, is isomorphic to $UT_n$, with an elementary $G$-grading.
\end{theorem}

We now recall the description of involutions on $UT_n$. The most important involution on $UT_n$ is $\circ: UT_n\longrightarrow UT_n$ given by $A^\circ=SA^tS$, where $S=\sum_{i}e_{i,n+1-i}$.  It is not difficult to see that $A^\circ  \in UT_n$ and that $\circ$ is an involution on $UT_n$. Also one can easily see that $e_{ij}^\circ=e_{n+1-i,n+1-j}$, i.e., $A^\circ$ is the reflection along to the secondary diagonal of $A$. 

If an invertible matrix $T\in UT_n$ satisfies $T^\circ=\pm T$, then $A\mapsto T^{-1}A^\circ T$ is also an involution on $UT_n$. An important particular case is when $n=2m$ is even and $T=J=\begin{pmatrix}
I_m & 0\\
0 & -I_m\\
\end{pmatrix}$, where $I_m$ denotes the $m\times m$ identity matrix. Then we obtain the involution $s:UT_n\longrightarrow UT_n$ given by $A^s=JA^\circ J$, which is also important in the description of involutions on $UT_n$ as we can see in the next theorem.

\begin{proposition}\cite[Proposition 2.5]{DiVKoshLaS}
	Let $F$ be an arbitrary field of characteristic different from 2 and let $UT_n$ be the algebra of upper triangular matrices over $F$. Let $:UT_n\longrightarrow UT_n$ be an involution on $UT_n$. Then $(UT_n,*)$ as an algebra with involution is isomorphic either to $(UT_n,\circ)$ or to $(UT_n,s)$.	
\end{proposition}

The main result of \cite{Valenti-Zaicev} is the following theorem:

\begin{theorem}\cite[Theorem 5.4]{Valenti-Zaicev}
	Let $F$ be an algebraically closed field of characteristic zero and let $R=UT_n=\oplus_{g\in G} R_g$ be the algebra of $n\times n$ upper triangular matrices over $F$ graded by a finite abelian group $G$. Suppose $R$ is endowed with a a graded involution $*$. Then $R$, as a $G$-graded algebra with involution is isomorphic to $UT_n$ with an elementary $G$-grading defined by an $n$-tuple $\overline{g}=(g_1,\dots,g_n)$ such that $g_1g_n=g_2g_{n-1}=\cdots=g_ng_1$ and with the involution $\circ$ or $s$. The involution $s$ can occur only if $n$ is even.
\end{theorem}

In a similar way in this section we give a description of degree-inverting involutions of $UT_n$. The methods used here are similar to those in \cite{Valenti-Zaicev}. In particular, we use the duality between $G$-gradings and $\hat G$-actions, where $\hat G$ is the dual group of $G$. For more details, see \cite{Giambruno-Zaicev}.

Let $G$ be a finite abelian group and let $\hat G$ be its dual group, i.e., the set of all irreducible characters $\lambda:G\longrightarrow F^*$ with pointwise multiplication, i.e., if $\lambda, \chi\in \hat G$, then  $(\lambda\chi)(g)=\lambda(g)\chi(g)$, for all $g\in G$. It is well known that $\hat G \cong G$ and the notions of $G$-grading and $\hat G$-action are equivalent in the following way. 

Suppose $R=\oplus_{g\in G} R_g$ is a $G$-grading on the algebra $R$. If $\lambda\in \hat G$, we define the map (also denote by $\lambda$) \[\begin{array}{cccc}\lambda: & R & \longrightarrow & R\\
											& \sum_{g\in G} a_g & \longmapsto & \sum_{g\in G}\lambda(g)a_g\\ \end{array}\]

One can easily see that the map $\lambda$ defined above is an automorphism of $R$. Conversely, if $\hat G$ acts on $R$ by automorphisms, by setting \[R_g=\{a\in R\,|\, \lambda(a)=\lambda(g)a, \forall \lambda \in \hat G\}\]
one obtains that $R=\oplus_{g\in G} R_g$ is a $G$-grading on $R$.

Now let $R=\oplus{g\in G}$ be an elementary $G$-grading on $R=UT_n$ defined by an $n$-tuple $\overline{g}=(g_1,\dots,g_n)\in G^n$. If $\lambda \in \hat G$, let $T_\lambda=\diag\{\lambda(g_1),\dots,\lambda(g_n)\}$. Then one can easily verify that for all $\lambda \in \hat G$, $\lambda (X)=T_{\lambda}^{-1}XT_\lambda$, for all $X\in UT_n$. 

With the above notation we have the following fact. 

\begin{proposition}\label{lambda}
	Let $R=\oplus_{g\in G} R_g$ be a finite dimensional algebra over an algebraically closed field of characteristic zero graded by a finite abelian group $G$.  Then $*:R\longrightarrow R$ is a degree-inverting involution on $R$ if and only if for all $g\in G$ and $a\in R_g$,\[(\lambda(a))^*=\lambda(g)^2\lambda(a^*).\]
\end{proposition}

\begin{proof}
	The involution $*$ is degree-inverting if and only if for all $g\in G$ and $a\in R_g$, $a^*\in R_{g^{-1}}$. So let $a\in R_g$. 
	
	If $*$ is degree-inverting, then $\lambda (a)= \lambda(g)a$ and $\lambda(a^*)=\lambda(g^{-1})a^*$. By applying $*$ in the first equation and comparing with the second, we obtain $\lambda(a)^*=\lambda(g)^2\lambda(a^*)$.
	
	Conversely, if the above equation holds, we obtain that $\lambda(a^*)=\lambda(g^{-1})a^*$, i.e., $a^*\in R_{g^{-1}}$.
\end{proof}

\begin{lemma}
	Let $B=\sum_{i\leq j}b_{ij}e_{ij}\in UT_n$ be an invertible matrix and let $D=\sum d_ie_{ii}\in UT_n$ be an invertible diagonal matrix. If $B^{-1}DB$ is a diagonal matrix, then $b_{ij}=0$ for all $(i,j)$ such that $d_i\neq d_j$.
\end{lemma}

\begin{proof}
	Write $B=(b_{ij})$ and let $B^{-1}DB=D'$ be a diagonal matrix. Then $B=D^{-1}B D'$. Write $D=\sum_{i=1}^n d_ie_{ii}$ and $D'=\sum_{i=1}^n d_i'e_{ii}$. The above matrix equation yields for all $i\leq j$, \[b_{ij}e_{ij}=d_i^{-1}d_j'b_{ij}e_{ij}\]
	Since $B$ is invertible, $b_{ii}\neq 0$, for all $i$. Put $i=j$ to get $d_i=d_i'$, for all $i$ and for all $i\leq j$,  \[b_{ij}=d_i^{-1}d_jb_{ij}\]
	From what we deduce that $d_i\neq d_j$ implies that $b_{ij}=0$.
	
\end{proof}

\begin{proposition}\label{homogeneous}
	Suppose $*$ is a degree-inverting involution on $UT_n$ and $B\in UT_n$ is such that $X^*=B^{-1}X^\circ B$, for all $X\in UT_n$. Then $B$ is homogeneous of degree $0$. 
\end{proposition}

\begin{proof}
	By Proposition \ref{lambda}, we know that $*$ is a degree-inverting involution if and only if for all $g\in G$, $X\in R_g$ and $\lambda \in \hat G$, we have $\lambda(X)^*=\lambda(g)^2\lambda(X^*)$.
	The latter is equivalent to 
	\[(T_\lambda^{\circ -1}B)^{-1}X^\circ T_\lambda^{\circ -1}B=(BT_\lambda)^{-1}\lambda(g)^2 X (BT_\lambda)\]
	In particular, for $g=e$, $\lambda(g)=0$. Since any diagonal matrix lies in $R_e$ and any diagonal matrix is of the form $D^\circ$ for some diagonal matrix $D$, it follows that 
	\[(T_\lambda^{\circ -1}B)^{-1}D T_\lambda^{\circ -1}B=(BT_\lambda)^{-1}\lambda(g)^2 D (BT_\lambda)\]
	for all diagonal matrix $D$. In particular, for all diagonal matrix  $D$,
	\[BT_\lambda B^{-1}T_\lambda^\circ D= D BT_\lambda B^{-1}T_\lambda^\circ\]
	i.e., $BT_\lambda B^{-1}T_\lambda^\circ$ lies in the centralizer of the diagonal matrices in $UT_n$, which is the set of diagonal matrices. Hence, there exists $\tilde D$ diagonal such hat $BT_\lambda B^{-1}T_\lambda^\circ=\tilde D$ and then \[BT_\lambda B^{-1}=\tilde D T_\lambda^{\circ -1}\] which is still a diagonal matrix, since $T_\lambda^{\circ -1}$ is so. Since the above equation holds for all $\lambda\in \hat G$, if $g_i\neq g_j$, there exist $\lambda \in \hat G$ such that $\lambda(g_i)\neq \lambda(g_j)$. Then by the previous lemma, $b_{ij}=0$ and we conclude that $B$ is homogeneous of degree $e$.
\end{proof}

\begin{corollary}\label{iffcirc}
	The involution $*$ is degree-inverting on $UT_n$ if and only if the involution $\circ$ is so.
\end{corollary}

By the above corollary, it is enough to analyze when $\circ$ is a degree inverting involution.

\begin{proposition}\label{circ}
	Let $UT_n=R=\oplus_{g\in G} R_g$ with an elementary grading defined by $\overline{g}=(g_1,\dots,g_n)$. Then $\circ:R\longrightarrow R$ is a degree inverting involution on $R$ if and only if $g_1^{-1}g_n=g_2^{-1}g_{n-1}=\cdots=g_n^{-1}g_1$.
\end{proposition}

\begin{proof}
	Of course $e_{ij}^\circ =e_{n+1-j,n+1-i}$. Hence $\deg(e_{ij}^\circ)=g_{n+1-j}^{-1}g_{n+1-i}$. Since $\deg(e_{ij})=g_i^{-1}g_j$, the involution $*$ is degree-inverting if and only if for all $i\leq j$, we have  $g_{n+1-j}^{-1}g_{n+1-i}=g_j^{-1}g_i$. The later is equivalent to $g_i^{-1}g_{n+1-i}=g_j^{-1}g_{n+1-j}$ for all $i\leq j$ which proves the proposition.
\end{proof}

We need the following result of \cite{DiVKoshLaS}.

\begin{lemma}\cite[Lemma 2.4]{DiVKoshLaS}\label{inv.utn}
	Let $D$ be a non-singular matrix of $UT_n$ and $D^*=D$ where $*=\circ$ or $s$.
	\begin{enumerate}
		\item If $n=2m$ and 
		$D=\begin{pmatrix}
		X & Y \\
		0 & Z
		\end{pmatrix}$, where $X,Z\in UT_m(F)$ and $Y\in M_m(F)$, then $D$ can be decomposed as $D=CC^*$, where 
		$C=\begin{pmatrix}
		E_m & \frac{1}{2}Y\\
		0 & Z\\
		\end{pmatrix}$ with $E_m$ the $m\times m$ identity matrix.
		\item If $n=2m+1$ and $D=\begin{pmatrix}
		X & a & Y \\
		0 & 1 & b \\
		0 & 0 & Z 
		\end{pmatrix}$ with $X,Z\in UT_m(F)$, $Y\in M_m(F)$, $a\in M_{m\times 1}(F)$ and $b\in M_{1\times m}(F)$, then $*=\circ$ and $D$ can be decomposed as $D=CC^*$ where $C=\begin{pmatrix}
		E_m & 0 & \frac{1}{2} Y\\
		0 & 1 & b\\
		0 & 0 & Z
		\end{pmatrix}$.
	\end{enumerate}
\end{lemma}

From the proof of Theorem 5.4 of \cite{Valenti-Zaicev} one can deduce the following result. For the sake of completeness, we present this proof here.

\begin{proposition}\label{iso}
	Let $F$ be an algebraically closed field of characteristic zero and let $R=UT_n=\oplus_{g\in G}R_g$ be the algebra of $n\times n$ upper-triangular matrices over $F$, graded by a finite abelian group $G$. Suppose $*:R\longrightarrow R$ is an involution on $R$ given by $X^*=B^{-1}X^\circ B$, for some nonsingular matrix $B\in R_e$. Then $R$ is isomorphic as a graded algebra with involution, to $UT_n$ with an elementary $G$-grading and with the involution $\circ$ or $s$. The involution $s$ can occur only if $n$ is even.
\end{proposition}
 
\begin{proof}
	By Theorem \ref{gradingUTn}, we may assume that the $G$-grading is elementary. Since $*$ is an involution, $B$ satisfies $B^\circ=B$ or $B^\circ=-B$.
	
	Suppose first that $B^\circ=B$. Since $B$ is invertible and $B\in R_e$, we also have $B^{-1}\in R_e$ and $(B^{-1})^\circ=B^{-1}$. Note that $B$ is defined up to a scalar multiple. Then we can apply Lemma \ref{inv.utn} to $D=B^{-1}$ assuming $d_{m+1,m+1}=1$ in case $n=2m+1$. Hence $D=CC^\circ$, and we since $D\in R_e$, from Lemma \ref{inv.utn}, it follows that $C\in R_e$. Replacing $B$ with $(C^{\circ})^{-1}C^{-1}$, we get $X^*=CC^\circ X^\circ (C^\circ)^{-1}C^{-1}$ and then $C^{-1}X^*C=C^{\circ}X^\circ(C^\circ)^{-1}=(C^{-1}XC)^\circ$.
	
	Let now $\varphi:R\longrightarrow R$ be defined by $\varphi(X)=C^{-1}XC$. Then $\varphi$ is an automorphism of $R$ preserving the $G$-degree, since $C\in R_e$. Moreover, by the above considerations, it follows that $\varphi(X^*)=\varphi(X)^\circ$, i.e., $\varphi$ is an isomorphism of graded algebras with involution. This completes the proof in case $B$ is a symmetric matrix with respect to $\circ$.
	
	Now suppose $B^\circ=-B$. Then $n=2m$ is even, since $B$ is non-singular. We follow the argument of the proof of \cite[Proposition 2.5]{DiVKoshLaS}. Namely, for $J=\begin{pmatrix}
	E_m & 0 \\
	0 & -E_m
	\end{pmatrix}$, we can write 
	
	\begin{align*}
	X^* & 	=  B^{-1}X^\circ B\\
		&	=  B^{-1}J^{-1}JX^\circ J^{-1}JB\\
		&	=  (JB)^{-1}JX^\circ J{-1}(JB)\\
		&	=(JB)^{-1}X^s(JB)
	\end{align*}
	 and $(JB)^s=B^sJ^s=JB^\circ J^{-1}(-J)=-JB^\circ=JB$. Hence, we can apply Lemma \ref{inv.utn} to $D=(JB)^{-1}$. As before, $D=CC^s$, and $\varphi(X)=C^{-1}XC$ satisfies $\varphi(X^*)=\varphi(X)^s$, i.e., $\varphi$ is an isomorphism of graded algebras with involution.
\end{proof}
 
 As a consequence, we obtain the main result of this section.
 
\begin{corollary}
	Let $F$ be an algebraically closed field of characteristic zero and let $R=UT_n=\oplus_{g\in G} R_g$ be the algebra of $n\times n$ upper-triangular matrices over $F$, graded by a finite abelian group $G$. Suppose $R$ is endowed with a degree-inverting involution $*$. Then $R$ as a graded algebra with involution is isomorphic to $UT_n$ with an elementary $G$-grading defined by an $n$-tuple $\overline{g}=(g_1,\dots,g_n)$ such that $g_1^{-1}g_n=g_2^{-1}g_{n-1}=\cdots=g_n^{-1}g_1$ and with the involution $\circ$ or $s$. The involution $s$ can occur only if $n$ is even.
\end{corollary}

\begin{proof}
	First notice that if $*=\circ$ and the elementary grading induced by $(g_1,\dots, g_n)$ satisfies  $g_1^{-1}g_n=g_2^{-1}g_{n-1}=\cdots=g_n^{-1}g_1$, then by Proposition \ref{circ} $\circ$ is a degree-inverting involution. Since $X^s=J^{-1}X^\circ J$, and $\deg(J)=e$, $s$ is degree inverting if and only if $\circ$ is so. So if $*=s$ and the grading satisfies $g_1^{-1}g_n=g_2^{-1}g_{n-1}=\cdots=g_n^{-1}g_1$, then $s$ is a degree-inverting involution.
	
	Let now $R=\oplus_{g\in G}R_g$ be a $G$-grading on $R$. If $*$ is a degree-inverting involution on $R$, there exists a nonsingular matrix $B$ such that $X^*=B^{-1}X^\circ B$, for all $X\in UT_n$. By Proposition \ref{homogeneous}, $\deg(B)=0$ and by Proposition \ref{iso}, $R$ is isomorphic as a graded algebra with involution, to $UT_n$ with an elementary grading defined by an $n$-tuple $\overline{g}=(g_1,\dots,g_n)\in G^n$ and with the involution $\circ$ or $s$. The involution $s$ can occur only if $n$ is even. In any case, the involution satisfies $g_1^{-1}g_n=g_2^{-1}g_{n-1}=\cdots=g_n^{-1}g_1$.	
\end{proof}

\begin{flushleft}
	\textbf{Acknowledgements}
\end{flushleft}
This work was started when the authors were visiting IMPA in (brazilian) summer 2016. The authors would like to thank IMPA for the hospitality and for the financial support.

\end{document}